\documentclass[12pt]{article}
\usepackage{amsmath}
\usepackage{amssymb}
\usepackage{amsthm}
\providecommand{\abs}[1]{\lvert#1\rvert}

\newcommand{\ov}{\overline}
\newtheorem{thm}{Theorem}

\newtheorem{claim}{Claim}
\newcommand{\dist}{\mbox{\em dist\/}}

\makeatletter
\newcounter{subclaim}
\newenvironment{subclaim}{\stepcounter{subclaim}\vspace{-12pt}\begin{itemize}\item[(\roman{subclaim})]\def\@currentlabel{(\roman{subclaim})}}
{\end{itemize}\vspace{-12pt}
}
\let\@oldproof=\proof\def\proof{\setcounter{subclaim}{0}\@oldproof} % reset counter at beginning of proof                                     
\makeatother

\begin{document}

\begin{center}
{\bf\Large A De Bruijn--Erd\H os theorem\\ 
\rule{0pt}{16pt}for $1$-$2$ metric spaces}\\
\rule{0pt}{16pt} Va\v sek Chv\' atal (Concordia University, Montreal)\footnote{{\tt
     chvatal@cse.concordia.ca}\\
\phantom{xxxi}Canada Research Chair in Combinatorial
Optimization}\\
\end{center}

\begin{center}
{\bf Abstract}
\end{center}
\vspace{-0.3cm}
{\small A special case of a combinatorial theorem of De
  Bruijn and Erd\H os asserts that every noncollinear set of $n$
  points in the plane determines at least $n$ distinct lines. Chen and
  Chv\'{a}tal suggested a possible generalization of this assertion in
  metric spaces with appropriately defined lines. We prove this
  generalization in all metric spaces where each nonzero distance
  equals $1$ or $2$.}
\vspace{0.3cm}

It is well known that 
\begin{subclaim} \label{claim.dbe}
{\em
every noncollinear set of $n$ points in the plane
determines at least $n$ distinct lines.}
\end{subclaim}
As noted by Erd\H{o}s~\cite{E43}, theorem~(i) is a corollary of the
Sylvester--Gallai theorem (asserting that, for every noncollinear set
$S$ of finitely many points in the plane, some line goes through
precisely two points of $S$); it is also a special case of a
combinatorial theorem proved later by De Bruijn and Erd\H os
\cite{DE48}.

Chen and Chv\'{a}tal \cite{CC08} suggested that theorem
\ref{claim.dbe} might generalize in the framework of metric spaces.
In a Euclidean space, line $\ov{uv}$ is characterized as 
\begin{multline*}
\ov{uv}=\{p:
\dist(p,u)\!+\!\dist(u,v)\!=\!\dist(p,v)\;\mbox{{\rm or}}\;\\
\dist(u,p)\!+\!\dist(p,v)\!=\!\dist(u,v)\;\mbox{{\rm or}}\;  
\dist(u,v)\!+\!\dist(v,p)\!=\!\dist(u,p)\}, 
\end{multline*}
where $\dist$ is the Euclidean metric; in an arbitrary metric space
$(S,\dist)$, the same relation may be taken for the definition of the
line. With this definition of lines in metric spaces, Chen and
Chv\'{a}tal asked:
\begin{subclaim} \label{ccc}
{\em True or false? Every metric space on $n$ points, where $n\ge 2$,
  either has at least $n$ distinct lines or else has a line that
  consists of all $n$ points.}
\end{subclaim}
Let us say that a metric space on $n$ points has the {\em De Bruijn -
  Erd\H{o}s property\/} if it either has at least $n$ distinct lines
or else has a line that consists of all $n$ points: now we may state
\ref{ccc} by asking whether or not all metric spaces on at least $2$
points have the De Bruijn - Erd\H{o}s property. A survey of results
related to this question appears in~\cite{BBC}.

 By a {\em $1$-$2$ metric space,\/} we mean a metric space where each
 nonzero distance is $1$ or $2$. Chiniforooshan and Chv\'
 atal~\cite{CC11}  proved that
\begin{subclaim} \label{cc12}
{\em every $1$-$2$ metric space on $n$ points has $\Omega(n^{4/3})$
  distinct lines and this bound is tight.}
\end{subclaim}
This result states that all sufficiently large $1$-$2$ metric spaces
have a property far stronger than the De Bruijn - Erd\H{o}s property,
but it does not imply that all $1$-$2$ metric spaces on at least $2$
points have the De Bruijn - Erd\H{o}s property.  The purpose of the
present note is to remove this blemish.

\begin{thm}\label{only}
All $1$-$2$ metric spaces on at least $2$ points have the De Bruijn -
Erd\H{o}s property.
\end{thm}

The rest of this note is devoted to a proof of Theorem~\ref{only}.  A
key notion in the proof, one borrowed from \cite{CC11}, is the notion
of {\em twins\/} in a $1$-$2$ metric space: these are points $u,v$
such that $\dist(u,v)=2$ and $\dist(u,w)=\dist(v,w)$ for all points
$w$ distinct from both $u$ and $v$. Use of this notion in counting
lines is pointed out in the following claim (also borrowed from
\cite{CC11}), whose proof is straightforward.

\begin{claim}\label{c0}
If $u_1,u_2,u_3,u_4$ are four distinct points in a $1$-$2$ metric space, then\\
\mbox{\hspace{0.5cm}}$\bullet$ if $\dist(u_1,u_2)\ne \dist(u_3,u_4)$, 
then $\ov {u_1u_2}\ne \ov {u_3u_4}$,\\
\mbox{\hspace{0.5cm}}$\bullet$ if $\dist(u_1,u_2)=\dist(u_2,u_3)=2$, 
then $\ov {u_1u_2}\ne \ov {u_2u_3}$,\\
\mbox{\hspace{0.5cm}}$\bullet$ if $\dist(u_1,u_2)=\dist(u_2,u_3)=1$
and $u_1,u_3$ are not twins,\\ 
\mbox{\hspace{0.5cm}}$\phantom{\bullet}$
then $\ov {u_1u_2}\ne \ov {u_2u_3}$.
\end{claim}

By a {\em critical $1$-$2$ metric space,\/} we shall mean a smallest
counterexample to Theorem~\ref{only}; in a sequence of claims, we
shall gradually prove the nonexistence of a critical $1$-$2$ metric
space.  We shall say that a line in a metric space is {\em
  universal\/} if, and only if, it consists of all points of the
space.

\begin{claim}\label{c1}
For every pair $u,v$ of twins in a critical $1$-$2$ metric space,
there is a third point $w$ in this space such that
$\dist(u,w)=\dist(v,w)=2$ and $\dist(x,y)=1$ whenever $x\in\{u,v,w\}$,
$y\not\in\{u,v,w\}$.
\end{claim}

\begin{proof}
Let $S$ denote the space we are dealing with. Since $S$ is critical,
$S$ does not have the De Bruijn - Erd\H{o}s property and $S\setminus
u$ has the De Bruijn - Erd\H{o}s property. We will derive the
existence of $w$ from these two facts. 

The assumption that $u,v$ are twins implies that\\
\mbox{\hspace{0.5cm}}($a$) if $x,y$ are distinct points in $S\setminus\{u,v\}$,
then the line $\overline{xy}$ in $S$ contains either both $u,v$ or
neither of $u,v$;\\
\mbox{\hspace{0.5cm}}($b$) if $w\in S\setminus u$ and $\dist(w,v)=1$, then the line
$\overline{wv}$ in $S$ (and the line $\overline{wu}$ in $S$) contains both $u,v$;\\
\mbox{\hspace{0.5cm}}($c$) if $w\in S\setminus u$ and $\dist(w,v)=2$, then the line
line $\overline{wv}$ in $S$ contains $v$ and not $u$ 
and the line $\overline{wu}$ in $S$ contains $u$ and not $v$.

Since $S$ does not have the
De Bruijn - Erd\H{o}s property, we have $\overline{uv}\ne S$; since
$u$ and $v$ are twins, it follows that\\
\mbox{\hspace{0.5cm}}($d$) there is a $w$ in $S\setminus u$ such that 
$\dist(w,v)=2$.\\
From ($a$), ($b$), ($c$), ($d$), we conclude that\\
\mbox{\hspace{0.5cm}}($e$) the number of lines in $S$ exceeds the
number of  lines in $S\setminus u$.\\
Since $S$ does not have the
De Bruijn - Erd\H{o}s property, the number of lines in $S$ is less
than $\abs{S}$, and so ($e$) implies that the number of lines in
$S\setminus u$ is less than 
$\abs{S\setminus u}$; since $S\setminus
u$ has the De Bruijn - Erd\H{o}s property, it follows that\\ 
\mbox{\hspace{0.5cm}}($f$) $S\setminus u$ has a universal line.\\ 
Since $S$ does not have the De Bruijn - Erd\H{o}s property,\\
\mbox{\hspace{0.5cm}}($g$) $S$ has no universal line.\\
Facts ($a$), ($f$), and ($g$) together imply that some line
$\overline{wv}$ in $S\setminus u$ is universal.  
Now ($b$) and ($g$) together imply that 
$\dist(w,v)=2$; since $u,v$ are twins, it follows that 
$\dist(u,v)=2$ and $\dist(w,u)=2$. 
Since $\overline{wv}$ is a universal line in
$S\setminus u$, we have $\dist(w,y)=\dist(v,y)=1$ whenever
$y\not\in\{u,v,w\}$; since $u,v$ are twins, it follows that 
$\dist(u,y)=1$ whenever
$y\not\in\{u,v,w\}$. 
\end{proof}

\begin{claim}\label{c2}
No critical $1$-$2$ metric space contains a pair of twins. 
\end{claim}
\begin{proof}
Assume the contrary: some critical $1$-$2$ metric space $S$ contains a
pair of twins.  We will show that $S$ has at leat $\abs{S}$ lines,
contradicting the assumption that $S$ does not have the De Bruijn -
Erd\H{o}s property.  For this purpose, consider the largest set
$\{T_1,T_2,\ldots ,T_k\}$ of pairwise disjoint three-point subsets of
$S$ such that $\dist(u,v)=2$ whenever $u,v$ are distinct points in the
same $T_i$ and such that $\dist(u,x)=1$ whenever $u\in T_i$, $x\not\in
T_i$ for some $i$.  Since $S$ contains a pair of twins, Claim~\ref{c1}
guarantees that $k\ge 1$; we will derive the existence of $\abs{S}$
lines in $S$ from this fact.

Let ${\cal L}_1$ denote the set of all lines $\overline{uv}$ such that
$u,v$ are distinct points in the same $T_i$. If 
$\overline{uv}\in {\cal L}_1$, then $\overline{uv}=S\setminus w$,
where $\{u,v,w\}=T_i$ for some $i$; it follows that\\ 
\mbox{\hspace{0.5cm}}($a$) ${\cal L}_1$ consists of the $3k$ sets
$S\setminus w$ with $w$ ranging through $\cup_{i=1}^k T_i$.\\
Next, choose a point $r$ in
$T_1$ and let ${\cal L}_2$ denote the set of all lines $\overline{rx}$
such that $x\in S\setminus \cup_{i=1}^k T_i$. 
Claim~\ref{c1} and maximality of $k$ together guarantee that $S$
contains no pair $x,y$ of twins such that $x,y\in S\setminus \cup_{i=1}^k
T_i$. This fact and Claim~\ref{c0} together imply
that\\ 
\mbox{\hspace{0.5cm}}($b$) $\abs{{\cal L}_2}=\abs{S}-3k$.\\
Finally, note that each line in ${\cal L}_2$ includes all points of
$T_1$ and no points of $T_2$. This observation and ($a$) together
imply that ${\cal L}_1\cap {\cal L}_2=\emptyset$, and so 
$\abs{{\cal L}_1\cup {\cal L}_2}=\abs{S}$ by ($a$) and ($b$).
\end{proof}

Each $1$-$2$ metric space can be thought of as a complete graph with
each edge $uv$ labeled by $\dist(u,v)$. Given edges $uv, xy$ of this
complete graph, let us write $uv\approx xy$ to mean that
$\overline{uv}=\overline{xy}$.  The following fact is a direct
consequence of Claim~\ref{c0} combined with Claim~\ref{c2}.
\begin{claim}\label{c3}
Each equivalence class of the equivalence relation $\approx$ in a critical $1$-$2$
metric space is a set of pairwise disjoint edges with identical
labels or else a (not necessarily proper) subset of a cycle of length
four with alternating labels.
\end{claim}

\begin{claim}\label{c4}
The size of each equivalence class of the equivalence relation $\approx$ in a critical $1$-$2$
metric space on $n$ points is at most $\max\{(n-1)/2, 4\}$.
\end{claim}
\begin{proof}
This is a direct corollary of Claim~\ref{c3} combined with the
observation that an equivalence class of $n/2$ pairwise disjoint edges
defines a universal line.
\end{proof}

\begin{claim}\label{c5}
Every critical $1$-$2$ metric space has at most $7$ points.
\end{claim}
\begin{proof}
Consider an arbitrary critical $1$-$2$ metric space and let $n$ denote
the number of its points.  Since this space does not have the De
Bruijn - Erd\H{o}s property, it has fewer than $n$ lines, and so its
equivalence relation $\approx$ partitions the $n(n-1)/2$ edges of its
complete graph into at most $n-1$ classes. Since the largest of these
classes has size at least $n/2$, Claim~\ref{c4} implies that $n/2\le 
\max\{(n-1)/2, 4\}$, and so $n\le 8$. If $n=8$, then the $28$ edges of
the complete graph are partitioned into $7$ equivalence classes of
size $4$. By Claim~\ref{c3}, each of these equivalence classes is a
cycle of length four. But this is impossible, since the edge set of the
complete graph on eight vertices cannot be partitioned into cycles:
each vertex of this graph has an odd degree.
\end{proof}

\begin{claim}\label{c6}
No critical $1$-$2$ metric space has $7$ points.
\end{claim}
\begin{proof}
Consider an arbitrary critical $1$-$2$ metric space on $7$
points. Since this space does not have the De
Bruijn - Erd\H{o}s property, it has fewer than $7$ lines, and so its
equivalence relation $\approx$ partitions the $21$ edges of its
complete graph into at most $6$ classes. By Claim~\ref{c4}, each of
these classes has size at most $4$, and so at least three of them have
size precisely $4$; by By Claim~\ref{c3}, each of these three classes
is a cycle of length four. Let $G_1,G_2,G_3$ denote these three
subgraphs of the complete graph on seven vertices.

Since $G_1,G_2,G_3$ are pairwise edge-disjoint, every two of them
share at most two vertices; since their union has only seven vertices,
some two of them share at least two vertices; we may assume (after a
permutation of subscripts if necessary) that $G_1$ and $G_2$ share
precisely two vertices. Let us name these two vertices $u,v$. Since
$G_1$ and $G_2$ are edge-disjoint, we may assume (after a switch of
subscripts if necessary) that vertices $u,v$ are adjacent in $G_1$
and nonadjacent in $G_2$.

Next, we may name $w,x$ the remaining two vertices in $G_1$ in such a
way that the four edges of $G_1$ are $uv,vw,wx,ux$; we may name $y,z$
the remaining two vertices in $G_2$ in such a way that the four edges
of $G_2$ are $uy,uz,vz,vy$.  Since the labels on the edges of $G_2$
alternate, we may assume (after switching $y$ and $z$ if necessary)
that $\dist(u,y)=1$, $\dist(u,z)=2$, $\dist(v,z)=1$, $\dist(v,y)=2$.
Since $\ov{uy}=\ov{vy}$, we have $u\in\ov{vy}$; since $\dist(v,y)=2$,
it follows that $\dist(u,v)=1$. In turn, since the labels on the edges
of $G_1$ alternate, we have $\dist(v,w)=2$, $\dist(w,x)=1$,
$\dist(u,x)=2$.

Now $\dist(y,u)+\dist(u,v)=\dist(y,v)$, and so $y\in \ov{uv}$; since
$uv\approx vw$, it follows that $y\in \ov{vw}$. But this is
impossible, since $\dist(v,w)=2$ and $\dist(v,y)=2$.
\end{proof}

\begin{claim}\label{c7}
Every critical $1$-$2$ metric space on $5$ or $6$ points 
contains points $u,v,w,x,y$ such that 
\begin{gather*}
\dist(u,w) = \dist(u,x) = \dist(v,w) = \dist(v,x) = 1,\\
 \dist(u,v) = \dist(w,x) = 2,\\
\dist(u,y) \ne \dist(v,y),\;\dist(w,y) \ne \dist(x,y).
\end{gather*}
\end{claim}
\begin{proof}
Consider an arbitrary critical $1$-$2$ metric space on $n$ points
such that $n=5$ or $n=6$. Since this space does not have the De Bruijn
- Erd\H{o}s property, it has fewer than $n$ lines, and so its
equivalence relation $\approx$ partitions the $n(n-1)/2$ edges of its
complete graph into at most $n-1$ classes. Since the largest of these
classes has size at least $3$, Claim~\ref{c3} and the absence of a
universal line together imply that there are points $u,v,w,x,y$ such
that 
\[
\dist(u,v)=2, \dist(v,w)=1, \dist(w,x)=2 \;\mbox{ and }\; \ov{uv}=\ov{vw}=\ov{wx}
\]
or else
\[
\dist(v,w)=1, \dist(w,x)=2, \dist(u,x)=1 \;\mbox{ and }\; \ov{vw}=\ov{wx}=\ov{ux}.
\]
In both cases, equality of the three lines implies that
\begin{gather*}
\dist(u,w) = \dist(u,x) = \dist(v,w) = \dist(v,x) = 1,\\
 \dist(u,v) = \dist(w,x) = 2.
\end{gather*}
Since $w,x$ are not twins, there is a point $y$ distinct from both of
them and such that $\dist(w,y) \ne \dist(x,y)$; we will complete the
proof by showing that $\dist(u,y) \ne \dist(v,y)$. 

To do this, assume the contrary: $\dist(u,y) = \dist(v,y)$.  Since
$y\not\in \ov{wx}$ and $\ov{vw}=\ov{wx}$, we have $y\not\in \ov{vw}$,
and so $\dist(v,y)=\dist(w,y)$. Now $\dist(u,y) \ne \dist(x,y)$, and
so $y\in\ov{ux}$; since $y\not\in \ov{wx}$, we cannot have
$\ov{vw}=\ov{wx}=\ov{ux}$, and so we must have
$\ov{uv}=\ov{vw}=\ov{wx}$. In particular, $y\not\in\ov{uv}$; since
$\dist(u,y) = \dist(v,y)$, we conclude that
\[
\dist(u,y) = \dist(v,y)=\dist(w,y)=2, \dist(x,y)=1.
\]
Since $u,v$ are not twins, there is a point $z$ distinct from both of
them and such that $\dist(u,z) \ne \dist(v,z)$; it follows that 
$\dist(x,z)$ is distinct from one of $\dist(u,z)$, $\dist(v,z)$, and
so $z$ belongs to one of the lines $\ov{ux}, \ov{vx}$. But then this
line is universal, a contradiction. 
\end{proof}

\begin{claim}\label{c8}
No critical $1$-$2$ metric space has $5$ or $6$ points.
\end{claim}
\begin{proof}
Consider an arbitrary critical $1$-$2$ metric space on $n$ points such
that $n=5$ or $n=6$ and let $u,v,w,x,y$ be as in Claim~\ref{c7}. We
may assume (after a cyclic shift of $u,w,v,x$ if necessary) that
\begin{gather*}
\dist(u,w) = \dist(u,x) = \dist(v,w) = \dist(v,x) = 1,\\
 \dist(u,v) = \dist(w,x) = 2,\\
\dist(u,y)=\dist(w,y)=1,\;
\dist(v,y)=\dist(x,y)=2.
\end{gather*}
Since 
\[
\ov{ux}\supseteq \{u,v,w,x,y\} \;\mbox{ and }\; 
\ov{vw}\supseteq \{u,v,w,x,y\},
\]
absence of a universal line implies that $n=6$ and that the sixth
point of our space lies outside the lines $\ov{ux}$ and $\ov{vw}$.
Let $z$ denote this sixth point. Since $z\not\in\ov{ux}$,
$z\not\in\ov{vw}$, we have $\dist(u,z)=\dist(x,z)$,
$\dist(v,z)=\dist(w,z)$, and so symmetry allows us to distinguish
between three cases:\\ 
\mbox{\hspace{0.5cm}} $\bullet$ $\dist(u,z)=\dist(x,z)=1$,
$\dist(v,z)=\dist(w,z)=1$,\\
\mbox{\hspace{0.5cm}} $\bullet$ $\dist(u,z)=\dist(x,z)=1$,
$\dist(v,z)=\dist(w,z)=2$,\\
\mbox{\hspace{0.5cm}} $\bullet$ $\dist(u,z)=\dist(x,z)=2$,
$\dist(v,z)=\dist(w,z)=2$.\\
Each of these three cases comprises two metric spaces, one with 
$\dist(y,z)=1$ and the other with $\dist(y,z)=2$. Altogether, there
are six metric spaces on six points to inspect; each of them has at least six lines.
\end{proof}

\begin{claim}\label{c9}
Every metric space on $2$, $3$, or $4$ points has the De Bruijn
- Erd\H{o}s property.
\end{claim}
\begin{proof}
Consider an arbitrary critical $1$-$2$ metric space on $n$ points. If
each of its lines has precisely $2$ points or if one of its lines has
precisely $n$ points, then this space has the De Bruijn - Erd\H{o}s
property; otherwise one of its lines has precisely $3$ points and
$n=4$. Let $T$ denote the $3$-point line and let $w$ denote the fourth
point of the space. If there are distinct $x,y$ in $T$ such that
$\ov{wx}=\ov{wy}$, then $\ov{xy}$ is a universal line; else the three
lines $\ov{wx}$ with $x$ ranging through $T$ are pairwise distinct
$2$-point lines.
\end{proof}

\begin{center}
{\bf Acknowledgement}
\end{center}
\vspace{-0.3cm}
This research was undertaken, in part, thanks to funding from the
Canada Research Chairs program.

\end{document}